\documentclass[reqno]{amsart}
\usepackage{latexsym}
\usepackage{amssymb}
\newtheorem*{Theo}{Theorem}
\newtheorem*{Lem}{Lemma}

\newcommand{\N}{\mathbb{N}}
\newcommand{\Q}{\mathbb{Q}}

\begin{document}
\title{The irrationality of a number theoretical series}
\author{J.-C. Schlage-Puchta}
\begin{abstract}
Denote by $\sigma_k(n)$ the sum of the $k$-th powers of the divisors
of $n$, and let $S_k=\sum_{n\geq 1}\frac{\sigma_k(n)}{n!}$. We prove
that Schinzel's conjecture H implies that $S_k$ is irrational, and
give an unconditional proof for the case $k=3$.
\end{abstract}
\maketitle

MSC-Index: 11A25, 11N36, 11J72

Let $\sigma_k(n)=\sum_{d|n} d^k$, and set $S_k=\sum_{n\geq 1}
\frac{\sigma_k(n)}{n!}$. For $k=0, 1$ it follows from a
general result by Erd{\H o}s and Straus \cite{ES}, that $S_k$ is
irrational, whereas for $k=2$ the same was shown by Erd{\H o}s and
Kac\cite{EK}. In \cite{Erdos}, Erd{\H o}s posed the question whether
$S_k$ is irrational for all $k$. We will prove the following theorem. 

\begin{Theo}
Define $S_k$ as above.
\begin{enumerate}
\item If Schinzel's conjecture H is true, then $S_k$ is
irrational for all $k\in\N$.
\item $S_3$ is irrational.
\end{enumerate}
\end{Theo}
Here, Schinzel's conjecture H is the following generalization of the
prime twin conjecture (cf. \cite{SSH}):\\
{\em Let $P_1, \ldots, P_k$ be integral polynomials with positive
leading coeficients, such that for each prime number $p$ there exists
some integer $a$ such that $P_1(a)\cdots P_k(a)\not\equiv 0\pmod {p}$.
Then there exist infinitely many integers $n$ such that $P_i(n)$ is prime for $1\leq
i\leq k$.}
\begin{proof}
Assume that $S_k$ was rational, say, $S_k=\frac{a}{b}$, $(a,
b)=1$. Then for every $n>b$, $(n-1)!S_k$ is an integer, and we deduce that
\[
\sum_{\nu\geq n}\frac{\sigma_k(\nu)}{(\nu)_{\nu-n+1}} \in\N,
\]
where $(x)_m=x(x-1)\cdots(x-m+1)$. Noting that for all $\varepsilon>0$
and $n$ sufficiently large, we have
$\sigma_k(n)<n^{k+\epsilon}$, we deduce that
\[
\left\|\sum_{\nu=n}^{n+k-1}\frac{\sigma_k(\nu)}{(\nu)_{\nu-n+1}}\right\|
< n^{-1+\epsilon}.
\]
Here and in the sequel, $\|x\|$ denotes the distance of $x$ to the
nearest integer. 
Now assume Schinzel's conjecture H, and fix some prime $p>k$. Then
there are infinitely many prime numbers $q\equiv 1\pmod{k!^k}$, such
that $\frac{q+i}{i+1}$ is prime for all $i\leq k$. For such a prime
number $q$ and $i\leq k$ we have
\[
\sigma_k(q+i) = \left(\left(\frac{q+i}{i+1}\right)^k+1\right)\sigma_k(i+1)
= q^k\sigma_{-k}(i+1) + O(1),
\]
hence, 
\[
\sum_{\nu=q}^{q+k-1}\frac{\sigma_k(\nu)}{(\nu)_{\nu-q+1}} =
\sum_{i=1}^{k}\sigma_{-k}(i)\frac{(q+i-1)^k}{(q+i-1)_i} + O(q^{-1}).
\]
The fraction $\frac{(q+i-1)^k}{(q+i-1)_i}$ can be written as
$P_{k, i}(q) + O(q^{-1})$ for some polynomial $P_{k, i}\in\Q[x]$,
combining our estimates we obtain that for 
all prime numbers $q\equiv 1\pmod{k!^k}$ with $\frac{q+i}{i+1}$ prime
for all $i\leq k$, we have
\begin{equation}
\label{eq:qprim}
\left\|\sum_{i=1}^{k}\sigma_{-k}(i)P_{k, i}(q)\}\right\| < q^{-1+\epsilon}.
\end{equation}
Now we repeat our argument, this time choosing an integer $q=pr$,
$q\equiv 1\pmod{k!^k}$, with $r$ prime, such that $\frac{q+i}{i+1}$ is
prime for all $i\leq k$. Arguing as above we deduce that 
\begin{equation}
\label{eq:qnichtprim}
\left\|\sigma_{-k}(p) P_{k, 1}(q) + \sum_{i=2}^{k}\sigma_{-k}(i)P_{k, i}(q)\}\right\| < q^{-1+\epsilon}.
\end{equation}
Since $q$ is fixed $\pmod{k!^k}$, the fractional part of
$\sigma_{-k}(i)P_{k, i}(q)$ does not depend on $q$, hence, comparing
(\ref{eq:qprim}) and (\ref{eq:qnichtprim}), we deduce that
\[
\left\|\sigma_{-k}(p) P_{k, 1}(q_1) - \sigma_{-k}(1) P_{k, 1}(q_2)\right\|
< q_1^{-1+\epsilon}
\]
holds true for all integers $q_1<q_2$, such that $q_1$ is $p$ times a
prime, $q_2$ is prime, $q_1\equiv q_2\equiv 1\pmod{k!^k}$, and
$\frac{q_j+i}{i}$ is prime for $j=1, 2$ and $i\leq k$. Using  the
fact that $P_{k, 1}(x)=x^{k-1}$ and $\sigma_{-k}(1)=1$, we obtain
\[
\left\|\frac{q_1^{k-1}}{p^k}\right\|<q_1^{-1+\epsilon}.
\]
For $q_1>p^2$, the left hand side cannot vanish, since then $p^2\nmid
q_1$. Hence, the left hand side is a nonzero rational number with
denominator dividing $p^k$, and therefore bounded below by
$p^{-k}$. However, $p$ is fixed, whereas $q_1$ may be chosen
arbitrary large, which yields a contradiction.

The proof of the second statement is similar, however, due to the fact
that we do not even know whether there is an infinitude of Sophie
Germain primes, it becomes more technical. As a substitute for
conjecture H we will use the following result. Denote by $P^-(n)$
the least prime factor of $n$.

\begin{Lem}
The number of primes $p\leq x$ such that
$P^-\left(\frac{p+1}{2}\right)$ and $P^-\left(\frac{p+2}{3}\right)$
are both greater then $x^{1/9}$ is $\gg\frac{x}{\log^3 x}$.
\end{Lem}
\begin{proof}
This follows from \cite[Theorem 7.4]{HR}.
\end{proof}
Note that the exponent $1/9$ is not optimal, however, it is sufficient
for our purpose. In the sequel, let $q$ be a prime number satisfying
$P^-\left(\frac{q+1}{2}\right)>q^{1/9}$ and
$P^-\left(\frac{q+2}{3}\right)>q^{1/9}$, and suppose that $q$ is
sufficiently large. As in the proof of the first
part of our theorem, we deduce that
\[
\left\|\frac{\sigma_3(q)}{q} + \frac{\sigma_3(q+1)}{q(q+1)} +
\frac{\sigma_3(q+2)}{q(q+1)(q+2)}\right\| < q^{-1+\epsilon}.
\]
By assumption we have $\sigma_3(q+2) = q^3 + \frac{q^3}{27} +
O(q^{8/3})$, that is, 
$\frac{\sigma_3(q+2)}{q(q+1)(q+2)} = \frac{28}{27}+O(q^{-1/3})$.
Moreover, denoting by $\{x\}$ the fractional part of the real number
$x$, we have $\left\{\frac{\sigma_3(q)}{q}\right\}=\frac{1}{q}$, and
we have
\[
\left\{\frac{\sigma_3(q+1)}{q(q+1)}-\frac{\sigma_3(q+1)}{(q+1)^2}\right\} =
1-\frac{1}{8} + O\left( \left\|\frac{(q+1)^2
}{q}\right\|\right) +
O(q^{-1/3}) = \frac{7}{8} + O(q^{-1/3}).
\]
Hence, setting $n=\frac{q+1}{2}$, we find that there are
$\gg\frac{x}{\log^3 x}$ integers $n\leq x$ with the following
properties:
\begin{enumerate}
\item[(i)] We have
\[
\left\|\frac{9\sigma_3(n)}{4n^2}+\frac{19}{216}\right\|\ll n^{-1/3},
\]
\item[(ii)] $P^-(n)>n^{-1/9}$,
\item[(iii)] $2n-1$ is prime, and $P^-\big(\frac{2n+1}{3}\big)>n^{-1/9}$.
\end{enumerate}
We will obtain a contradiction by estimating the number of integers
$n$ with these properties from above. If there were as
many integers $n$ with these properties, there has to be some $k\leq
9$, such that there are $\gg\frac{x}{\log^3 x}$ integers $n$ with
these properties which have precisely $k$ prime factors. We may assume
that $n$ is squarefree, for otherwise $n$ was divisible by the square of
an integer $k\geq n^{1/9}$, and the number of integers $n\in[x, 2x]$
with this property is bounded above by
\[
\sum_{k\geq x^{1/9}}\left[\frac{2x}{k^2}\right] \ll x^{8/9},
\]
which is of negligible size.
Let $p_1<p_2<\ldots<p_k$ be the prime factors
of $n$. Set $[k]=\{1, \ldots, k\}$. Then divisors of $n$ correspond to
subsets $I$ of $[k]$, and inserting the definition of $\sigma_3$, we see
that condition (i) is equivalent to 
\[
\left\|\sum_{I\subseteq[k]}\frac{9\prod_{i\in I} p_i}{4\prod_{i\not\in I}
p_i^2}+\frac{19}{216}\right\| \ll n^{-1/3}.
\]
The summand $I=[k]$ corresponds to the trivial divisor $n$, which
contributes $\frac{9n}{4}$. Since for $n$ sufficiently large, $n$ has to
be odd by condition (ii), the contribution is $\pm\frac{1}{4}\pmod{1}$.
Hence, all integers satisfying (i) and (ii) also satisfy
\begin{equation}
\label{eq:diophcrit}
\left\|\sum_{I\subseteq[k]}\frac{9\prod_{i\in I} p_i}{4\prod_{i\not\in I}
p_i^2}+\frac{19}{216}\pm\frac{1}{4}\right\| \ll n^{-1/3},
\end{equation}
If $k=1$, then
$n=p_1$, and 
(\ref{eq:diophcrit}) becomes $\|\frac{9}{4p_1^2}+\frac{19}{216}\pm\frac{1}{4}\| \ll
p_1^{1/3}$, which is impossible for $n$ sufficiently large. If $k=2$,
(\ref{eq:diophcrit}) is equivalent to 
\[
\left\|\frac{9p_2}{4p_1^2} + \frac{19}{216}\pm\frac{1}{4}\right\|\ll (p_1p_2)^{-1/9}
\]
since $p_2>p_1>n^{1/9}$. For fixed $p_1$, all admissible $p_2<x$ are
contained in $\ll\frac{x}{p_1^3}+1$ intervals of length
$\ll p_1^{2-2/9}$ each, hence, the number of admissible $p_2$ is $\ll
xp_1^{-1-2/9}$. Summing over all $p_1>x^{1/9}$, we find that the
number of integers $n\leq x$ with two prime factors satisfying
(\ref{eq:diophcrit}) is bounded above by $x^{1-2/81}$. Hence, we may
assume that $k\geq 3$, in particular, we have $p_1<x^{1/3}$. We divide
the interval $[x^{1/9}, x^{1/3}]$ into $\ll\log x$ intervals of the form
$[y, 2y]$ and will now estimate the number of integers $n\leq x$
satisfying conditions (i)--(iii) together with $p_1\in[y, 2y]$.  
Set
\[
\alpha = \sum_{1\not\in I\subset[k]}\frac{9\prod_{i\in I}
p_i}{4\prod_{i\not\in I} p_i^2}.
\]
Note that our assumption implies $p_1^2<\alpha<x$.We now distinguish
two cases, depending on the relative size of $\alpha$ and $y$. Let $C$
be a constant to be determined later, and assume first that for each
integer $2\leq\ell\leq 9$ we have
\begin{equation}
\label{eq:alphay}
\alpha\not\in [y^\ell\log^{-C} x, y^\ell\log^C x].
\end{equation}
Then we rewrite (\ref{eq:diophcrit}) as
\[
\|\alpha p_1 + \frac{\alpha}{p_1^2} + \frac{19}{216}\pm\frac{1}{4}\|\ll n^{-1/3}.
\]
It suffices to show that the number of integers $n_1\in[y, 2y]$ satisfying 
\begin{equation}
\label{eq:diophcrit2}
\|\alpha n_1 + \frac{\alpha}{n_1^2} + \frac{19}{216}\|\ll x^{-1/4}
\end{equation}
is bounded above by $\frac{y}{\log^6 y}$. This quantity is at most $yx^{-1/4}+ D$,
where $D=D(\alpha, y)$ is the discrepancy of the sequence $(\alpha n_1 +
\frac{\alpha}{n_1^2})_{n_1\in[y, 2y]}$. Bounding the discrepancy using the
Erd{\H o}s-Tur{\'a}n-inequality (see e.g. \cite[Corollary 1.1]{Mont})we
obtain
\[
D \ll \frac{y}{H} + \sum_{h\leq H}\frac{1}{h}\left|\sum_{n=y}^{2y} e(h
f(n))\right|.
\]
for any parameter $H\geq 1$, where have set $f(n)=\alpha
n+\frac{\alpha}{n^2}$. To bound the exponential sum on the right hand
side, it suffices to use the simplest van der Corput-type estimates
(see e.g. \cite[Theorem 2.9]{GK}). If the
integer $2\leq\ell\leq 8$ is determined by means of the inequality $y^\ell\log^C
x<\alpha<y^{\ell+1}\log^{-C} x$, we have
\[
\frac{\log^C}{y}\ll f^{(\ell+1)}(x)\ll\frac{1}{\log^C x},\quad \forall
 x\in[y, 2y].
\]
For $\ell\geq 3$ we deduce
\begin{eqnarray*}
\sum_{n=y}^{2y} e(h f(n)) & \ll & y\big(h\alpha
f^{(\ell+1)}(y)\big)^{1/(4Q-2)} + h^{-1} f^{(\ell+1)}(y)^{-1}\\
 & \ll & hy \log^{-C/Q} x + y\log^{-C} x,
\end{eqnarray*}
where $Q=2^{\ell+1}$, and therefore
\begin{eqnarray*}
D & \ll & \frac{y}{H} + \sum_{h\leq H}\frac{1}{h}\left|\sum_{n=y}^{2y} e(h
f(n))\right|\\
 & \ll & \frac{y}{H} + Hy \log^{-C/Q} x.
\end{eqnarray*}
Setting $H=\log^7 x$ and $C=14Q\leq 2^{13}$, we obtain
$D\ll\frac{y}{\log^7 x}$, and therefore, for $x$ sufficiently large,
$D\leq\frac{y}{\log^6 y}$. Note that, apart from (\ref{eq:alphay}),
this estimate is independent of $\alpha$, which shows that there are
$\ll\frac{x}{\log^5 x}$ integers $n\leq x$ satisfying conditions
(i)--(iii) together with (\ref{eq:alphay}).

Now we consider the case
\begin{equation}
\label{eq:alphay2}
\alpha\in [y^\ell\log^{-C} x, y^\ell\log^C x]
\end{equation}
for some integer $2\leq\ell\leq 9$. Fix prime numbers
$x^{1/9}<p_2<\dots< p_k$, and a real number $y$ such that $yp_2\cdots
p_k<x$, such that (\ref{eq:alphay2}) is satisfied. The prime numbers
$p_2, \ldots, p_k$ can be chosen in $\ll\frac{x}{y\log x}$ ways, and
there are $\ll\log\log x$ intervals of the form $[y, 2y]$ to be
considered. For each fixed $p_2, \ldots, p_k$, the number of primes
$p_1\in[y, 2y]$ such that $p_1\cdots p_k$ satisfies condition (iii) is
$\ll\frac{y}{\log^3 x}$, thus, the total number of integers $n$
satisfying conditions (ii) and (iii) as well as
\[
p_1^\ell\log^{-C} x \leq \alpha\leq 2p_1^\ell\log^C x
\]
for some integer $\ell$ is $\ll\frac{x\log\log x}{\log^4 x}$. Hence, the total number of integers
$n\leq x$ satisfying conditions (i)--(iii) is bounded above by
$\mathcal{O}\Big(\frac{x\log\log x}{\log^4 x}\Big)$, which contradicts our lower bound
$\frac{x}{\log^3 x}$, proving our theorem.
\end{proof}

J.-C. Schlage-Puchta\\
Mathematisches Institut\\
Eckerstr. 1\\
79111 Freiburg\\
Germany
\end{document}